\documentclass[aop,preprint]{imsart}

\usepackage{amssymb,amsbsy,amsthm,amsmath,graphicx,epsfig}

\startlocaldefs

\newtheorem{thm}{Theorem}[section]
\newtheorem{lem}[thm]{Lemma}
\newtheorem{prop}[thm]{Proposition}
\newtheorem{cor}[thm]{Corollary}

\newtheorem{ques}[thm]{Question}
\newtheorem*{thm*}{Theorem}
\newtheorem*{cor*}{Corollary}
\theoremstyle{remark}

\newtheorem*{rmk}{Remark}

\renewcommand{\bf}[1]{\boldsymbol{#1}}
\renewcommand{\rm}[1]{\mathrm{#1}}
\renewcommand{\cal}[1]{\mathcal{#1}}

\newcommand{\bbR}{\mathbb{R}}

\newcommand{\rmH}{\mathrm{H}}

\newcommand{\rmb}{\mathrm{b}}
\renewcommand{\d}{\mathrm{d}}
\newcommand{\rme}{\mathrm{e}}

\renewcommand{\P}{\mathcal{P}}
\newcommand{\Q}{\mathcal{Q}}

\renewcommand{\a}{\alpha}

\newcommand{\eps}{\varepsilon}
\newcommand{\g}{\gamma}
\renewcommand{\l}{\lambda}

\newcommand{\ol}[1]{\overline{#1}}

\renewcommand{\to}{\longrightarrow}

\renewcommand{\phi}{\varphi}
\newcommand{\cov}{\mathrm{cov}}
\renewcommand{\Pr}{\rm{Prob}}

\newcommand{\DTC}{\mathrm{DTC}}

\newcommand{\rmD}{\mathrm{D}}

\endlocaldefs

\begin{document}

\begin{frontmatter}
	\title{The structure of low-complexity Gibbs measures on product spaces}
	\runtitle{Low-complexity Gibbs measures}

		\author{\fnms{Tim} \snm{Austin}\ead[label=e1]{tim@math.ucla.edu}}
		
		\runauthor{Tim Austin}
		
		\affiliation{University of California, Los Angeles}
		
		\address{UCLA Mathematics Department, Box 951555,\\ Los Angeles, CA 90095-1555, U.S.A.\\\printead{e1}}

\begin{abstract}
	Let $K_1$, \dots, $K_n$ be bounded, complete, separable metric spaces. Let $\l_i$ be a Borel probability measure on $K_i$ for each $i$. Let $f:\prod_i K_i \to \bbR$ be a bounded and continuous potential function, and let
	\[\mu(\d \bf{x})\ \propto\ \rme^{f(\bf{x})}\l_1(\d x_1)\cdots \l_n(\d x_n)\]
	be the associated Gibbs distribution.
	
	At each point $\bf{x} \in \prod_i K_i$, one can define a `discrete gradient' $\nabla f(\bf{x},\,\cdot\,)$ by comparing the values of $f$ at all points which differ from $\bf{x}$ in at most one coordinate.  In case $\prod_i K_i = \{-1,1\}^n \subset \bbR^n$, the discrete gradient $\nabla f(\bf{x},\,\cdot\,)$ is naturally identified with a vector in $\bbR^n$.

This paper shows that a `low-complexity' assumption on $\nabla f$ implies that $\mu$ can be approximated by a mixture of other measures, relatively few in number, and most of them close to product measures in the sense of optimal transport.  This implies also an approximation to the partition function of $f$ in terms of product measures, along the lines of Chatterjee and Dembo's theory of `nonlinear large deviations'.

An important precedent for this work is a result of Eldan in the case $\prod_i K_i = \{-1,1\}^n$.  Eldan's assumption is that the discrete gradients $\nabla f(\bf{x},\,\cdot\,)$ all lie in a subset of $\bbR^n$ that has small Gaussian width.  His proof is based on the careful construction of a diffusion in $\bbR^n$ which starts at the origin and ends with the desired distribution on the subset $\{-1,1\}^n$.  Here our assumption is a more naive covering-number bound on the set of gradients $\{\nabla f(\bf{x},\,\cdot\,):\ \bf{x} \in \prod_i K_i\}$, and our proof relies only on basic inequalities of information theory.  As a result, it is shorter, and applies to Gibbs measures on arbitrary product spaces.
	\end{abstract}


\begin{keyword}[class=MSC]
\kwd[Primary ]{60B99}
\kwd[; secondary ]{60G99}
\kwd{82B20}
\kwd{94A17}
\end{keyword}

\begin{keyword}
\kwd{nonlinear large deviations}
\kwd{Gibbs measures}
\kwd{gradient complexity}
\kwd{dual total correlation}
\kwd{mixtures of product measures}
\end{keyword}

\end{frontmatter}

\thispagestyle{plain}

\section{Introduction}

Let $(K_1,d_{K_1})$, \dots, $(K_n,d_{K_n})$ be nonempty, complete, separable metric spaces, all with diameter at most one.  Let $C_\rm{b}(K_i)$ and $\Pr(K_i)$ denote the spaces of bounded continuous functions and Borel probability measures on $K_i$, respectively, and similarly for other topological spaces.  Let $\|\cdot\|$ denote the uniform norm on any space of real-valued functions.  Let $\l_i \in \Pr(K_i)$ be a fixed reference measure for each $i$. 

Let $f:\prod_{i=1}^n K_i \to \bbR$ be a bounded continuous function, and call it the `potential'.  Let
\begin{equation}\label{eq:Gibbs-main}
\mu(\d\bf{x}) := \frac{1}{Z}\rme^{f(\bf{x})}\prod_{i=1}^n\l_i(\d x_i)
\end{equation}
be the resulting Gibbs measure, where $Z$ is the normalizing constant that makes $\mu$ a probability measure.  In the language of statistical mechanics, $Z$ is the `partition function' of $f$.

Inside $C_\rmb(\prod_i K_i)$ lies the vector subspace of functions that have the form
\begin{equation}\label{eq:sum-fn}
\bf{x} \mapsto f_1(x_1) + \dots + f_n(x_n)
\end{equation}
for some $f_1 \in C_\rmb(K_1)$, \dots, $f_n \in C_\rmb(K_n)$.  We call such functions \textbf{additively separable}.  The choice of $f_1$, \dots, $f_n$ representing the function in~\eqref{eq:sum-fn} is unique up additive constants.  If $f$ is the function in~\eqref{eq:sum-fn}, then its Gibbs measure factorizes as
\[\frac{1}{Z}\prod_{i=1}^n (\rme^{f_i(x_i)}\l_i(\d x_i)),\]
so it is a product measure.  In this special case we denote this Gibbs measure by $\xi_f$.  Similarly, if $g \in C_\rmb(K_i)$ for some $i$, then we denote by $\xi_g$ the Gibbs measure on $K_i$ constructed from $g$ and $\l_i$.

Consider again a general $f \in C_\rm{b}(\prod_i K_i)$.  In the main result of this paper, we assume that $f$ has `low complexity' relative to the subspace of additively separable functions, and deduce a structure theorem for $\mu$ in terms of mixtures of product measures.  In case $|K_i| = 2$ for each $i$, theorems of this kind originate in Chatterjee and Dembo's work~\cite{ChaDem16} introducing the theory of `nonlinear large deviations'.  Chatterjee and Dembo's main result gives an approximation to the partition function $Z$ under such an assumption on $f$.  More recently, works by Eldan~\cite{Eld--GWcomp} and Eldan and Gross~\cite{EldGro18} have uncovered the approximate structure of the measure $\mu$ itself, again in case $|K_i|=2$.  Our main theorem fits a similar template to Eldan's, but it applies to general product spaces and its proof is quite different.

To explain the relevant notion of `low complexity', fix a reference element $\ast_i \in K_i$ for each $i$.  Given $\bf{x} \in \prod_i K_i$, $i \in [n]$, and $y \in K_i$, we define
\[\partial_i f(\bf{x},y) := f(x_1,\dots,x_{i-1},y,x_{i+1},\dots,x_n) - f(x_1,\dots,x_{i-1},\ast_i,x_{i+1},\dots,x_n).\]
Thus, we replace the $i^{\rm{th}}$ coordinate of $\bf{x}$ twice, first with $y$ and then with $\ast_i$, and take the difference of the resulting values of $f$. This should be regarded as a discrete analog of the `partial derivative of $f$ in the $i^{\rm{th}}$ coordinate'.  Beware that the definition of $\partial_i f$ depends on the choice of the reference point $\ast_i$.  We suppress this dependence in our notation, but return to this point after the statement of the main theorem below.

We assemble these new functions $\partial_i f$ into the single function
\[\nabla f(\bf{x},\bf{y}) := \sum_{i=1}^n \partial_if(\bf{x},y_i) \qquad \Big(\bf{x},\bf{y} \in \prod_i K_i\Big).\]
For a fixed point $\bf{x} \in \prod_i K_i$, the function $\nabla f(\bf{x},\ \cdot\ )$ is additively separable. We refer to it as the \textbf{discrete gradient} of $f$ at the point $\bf{x}$.

If $f\in C_\rmb(\prod_i K_i)$ has the property that $\nabla f(\bf{x},\,\cdot\,)$ is the same function for every $\bf{x} \in\prod_i K_i$, then an easy exercise shows that $f$ itself is additively separable.  Beyond this case, we can make the weaker assumption that $f$ has relatively few different gradients $\nabla f(\bf{x},\,\cdot\,)$ as $\bf{x}$ varies in $\prod_i K_i$, at least up to small errors. This is the notion of `low complexity' that we need.

The statement of our main theorem involves two other concepts.  The first is a mode of approximation between measures on a product space, provided by an appropriate transportation metric.  Endow the product space $\prod_i K_i$ with the normalized Hamming average of the metrics $d_{K_i}$:
\[d_n(\bf{x},\bf{y}) := \frac{1}{n}\sum_{i=1}^n d_{K_i}(x_i,y_i) \qquad \Big(\bf{x},\bf{y} \in \prod_i K_i\Big).\]
If each $K_i$ is just a finite alphabet with the discrete metric, then $d_n$ is the usual Hamming metric, normalized to have diameter $1$.  Now define the transportation metric over $d_n$ by
\begin{equation}\label{eq:trans}
\ol{d_n}(\mu,\nu) := \inf_\l \int d_n(\bf{x},\bf{y})\,\l(\d\bf{x},\d\bf{y}) \qquad \Big(\mu,\nu \in \Pr\Big(\prod_i K_i\Big)\Big),
\end{equation}
where $\l$ ranges over all couplings of $\mu$ and $\nu$. This is a standard and well-studied construction of a metric on probability measures: see, for instance,~\cite[Section 11.8]{Dud--book}.

The second concept we need is a quantity which measures the multi-variate correlation in a joint distribution on an $n$-fold product space.  There are many such quantities, but the one we need is called `dual total correlation' or `DTC', which goes back to work of Han in information theory~\cite{Han75}.  The definition of DTC is recalled in Subsection~\ref{subs:Han} below.  The recent paper~\cite{Aus--TC+DTC} studies DTC in some depth.  According to the main result of that paper, a small value of $\DTC$ implies that the joint distribution is close in $\ol{d_n}$ to a mixture of relatively few product measures.  The relevance of DTC to the present paper is therefore not surprising, but in fact we do not call on that result from~\cite{Aus--TC+DTC} in our work below.  Rather, the proof of the main theorem in the present paper is relatively self-contained, and happens to yield an upper bound on the DTC of the Gibbs measure~\eqref{eq:Gibbs-main} as a by-product.

Our main theorem decomposes $\mu$ as a mixture in a very concrete way: by partitioning $\prod_i K_i$ and conditioning on the parts.  Given any probability space $(K,\mu)$ and measurable subset $P\subseteq K$ of positive measure, we write $\mu_{|P}$ for the conditioned measure $\mu(\,\cdot\,|\,P)$.

\begin{thm*}[A]
 Let $\cal{P}$ be a partition of $\prod_i K_i$ with the property that
\begin{equation}\label{eq:P-small}
	\|\nabla f(\bf{x},\,\cdot\,) - \nabla f(\bf{y},\,\cdot\,)\| = \sup_{\bf{z}}|\nabla f(\bf{x},\bf{z}) - \nabla f(\bf{y},\bf{z})| < \delta n
	\end{equation}
	whenever $\bf{x},\bf{y}$ lie in the same part of $\cal{P}$. Then:
	\begin{enumerate}
		\item[a)] $\DTC(\mu) < \rmH_\mu(\P) + \delta n$, where $\rmH_\mu(\P)$ is the Shannon entropy of the partition $\P$ according to $\mu$,
		\item[b)] we have
		\[\sum_{P\in \cal{P}}\mu(P)\cdot \int \rmD(\mu_{|P}\,\|\,\xi_{\nabla f(\bf{y},\,\cdot\,)})\ \mu_{|P}(\d \bf{y}) < \rmH_\mu(\cal{P}) + \delta n,\]
		where $\rmD$ denotes Kullback--Leibler divergence, and
		\item[c)] we have
		\[\sum_{P\in \cal{P}}\mu(P)\cdot \int \ol{d_n}(\mu_{|P},\xi_{\nabla f(\bf{y},\,\cdot\,)})\ \mu_{|P}(\d \bf{y}) < \sqrt{\frac{1}{2}\Big(\frac{\rmH_\mu(\cal{P})}{n} + \delta\Big)}.\]
	\end{enumerate}
\end{thm*}

In this theorem, part (c) follows directly from part (b) using Marton's transportation-entropy inequality (Proposition~\ref{prop:Marton} below) and H\"older's inequality.  We include both of these parts in the statement because both kinds of approximation by product measures have an intrinsic interest and potential applications.  But most of the proof of Theorem A goes towards parts (a) and (b).

The definition of $\nabla f$ depends on the choice of the reference points $\ast_i$.  However, if we write $\ast'_1$, \dots, $\ast'_n$ for an alternative choice of reference points and write $\nabla'f$ for the alternative discrete gradients that result, then these discrete gradients are related by the cocycle equation
\[\nabla'f(\bf{x},\bf{y}) = \nabla f(\bf{x},\bf{y}) - \nabla f(\bf{x},(\ast'_1,\dots,\ast'_n)).\]
As a result, if $\P$ satisfies~\eqref{eq:P-small} for the discrete gradients $\nabla f(\bf{x},\,\cdot\,)$, then it satisfies the analogous bounds for the discrete gradients $\nabla' f(\bf{x},\,\cdot\,)$ with $\delta n$ loosened to $2\delta n$.  For this reason, the particular choice of the reference points $\ast_i$ has little effect on Theorem A.

Theorem A is valuable in case $\rmH_\mu(\cal{P})$ is small compared to $n$.  This is implied if we have enough control on $|\P|$ itself: for instance, if we know that
\begin{equation}\label{eq:bdd-cplx}
\rm{cov}_{\delta n}\Big(\Big\{\nabla f(\bf{x},\,\cdot\,):\ \bf{x} \in \prod_i K_i\Big\},\|\cdot\|\Big) \leq \rme^{\eps n}
\end{equation}
for some small $\eps$, where $\rm{cov}_{\delta n}(\,\cdot\,,\|\cdot\|)$ denotes the smallest cardinality of a covering by sets of $\|\cdot\|$-diameter less than $\delta n$.

\begin{cor*}[A$'$]
	If $f$ satisfies~\eqref{eq:bdd-cplx}, then there are (i) a partition $\cal{P}$ of $\prod_i K_i$ into at most $\rme^{\eps n}$ parts and (ii) a selection of additively separable functions $g_P$ for $P \in \cal{P}$ such that
			\[\sum_{P \in \cal{P}}\mu(P)\cdot \rmD(\mu_{|P}\,\|\,\xi_{g_P}) < (\eps + \delta)n\]
and
		\[\sum_{P \in \cal{P}}\mu(P)\cdot \ol{d_n}(\mu_{|P},\xi_{g_P}) < \sqrt{(\eps + \delta)/2}.\]
		(Informally: `most of the mass of $\mu$ is on conditioned measures $\mu_{|P}$, $P \in \cal{P}$, that are close to products'.)
\end{cor*}

The argument from Theorem A to Corollary A$'$ is very short, but we include it explicitly after the proof of Theorem A in Section~\ref{sec:A-proof}.

The proof of Theorem A brings together four basic results of information theory.  All are well-known up to some routine manipulations, but Section~\ref{sec:princs} lays them out carefully.  Then Section~\ref{sec:A-proof} completes the proof of Theorem A.

Subsequent sections explore some consequences and related results.

First, Section~\ref{sec:mix-terms} gives an approximate description of \emph{which} product measures appear in the mixture promised by Corollary A$'$: see Proposition~\ref{prop:approx-fix-pt}.  The description takes the form of an approximate fixed point equation for those measures in terms of $f$.  It follows almost immediately from Theorem A itself. It is an analog of the main result in~\cite{EldGro18} for the case of the Hamming cube, although that paper provides various finer details and applications that we do not pursue here.

Next, Section~\ref{sec:part-fn-approx} turns Theorem A into an approximation for the normalizing constant $Z$ in~\eqref{eq:Gibbs-main}.  This topic is the original concern of~\cite{ChaDem16}, and re-appears among the consequences of Eldan's main result in~\cite{Eld--GWcomp}.


\subsection*{Some further comparison with previous works}

The simplest setting for Theorem A is a product of two-point alphabets.  The works of Chatterjee and Dembo~\cite{ChaDem16} and Eldan and Gross~\cite{Eld--GWcomp,EldGro18} are confined to that case.

In~\cite{ChaDem16}, the authors assume that $f$ is defined on the whole of $\bbR^n$, and then restrict it to the product subset $\{0,1\}^n$.  Then they quantify the `complexity' of $f$ in terms of its classical gradient in the sense of calculus, rather than the discrete quantity $\nabla f$ that we discuss above.  Their main result is an estimate on the normalizing constant $Z$, roughly in terms of a variational formula over product measures.  This work has now been generalized to other alphabets by Yan~\cite{Yan--nldp}.  For Yan's generalization, he retains the feature that the alphabets $K_i$ are subsets of some given Banach spaces $V_i$, and that the complexity of $f$ is measured by its Fr\'echet derivative as a function on $\prod_i V_i$.  It should be straightforward to move between this notion of `low complexity' and ours, but we do not explore this further here.

In~\cite{Eld--GWcomp}, Eldan also regards his state space as a subset of $\bbR^n$. He uses the subset $\{-1,1\}^n$, which is more convenient for his proofs.  Like the assumption in Theorem A, Eldan's `low-complexity' assumption applies directly to the discrete gradient $\nabla f$, not its relative from calculus.  But his approach uses the embedding $\{-1,1\}^n \subset \bbR^n$ in another very essential way.  It relies on a diffusion process in $\bbR^n$ which starts at the origin and ends with the desired distribution on $\{-1,1\}^n$.  This approach seems quite different from ours.  One can regard the main contribution of the present paper as an alternative proof of something like Theorem A which (i) is simpler in the special case $K_i = \{-1,1\}$ and (ii) generalizes easily to other product spaces.  However, at some points Eldan's estimates seem to be sharper than ours; we compare them more carefully in Subsection~\ref{subs:Eldan-compar}.  In a subsequent work~\cite{EldGro18} Eldan and Gross have characterized which product measures appear in Eldan's structure theorem in terms of the function $f$; Proposition~\ref{prop:approx-fix-pt} below is similar to their result.

Chatterjee and Dembo's results in~\cite{ChaDem16} were initially motivated by statistical physics or large deviations questions in certain highly symmetric models of random graphs.  Chatterjee gives a careful introduction to this area in his monograph~\cite{Cha--randomgraphsbook}, where Chapter 8 is given to the `nonlinear large deviations approach'.  After applying the basic theory of nonlinear large deviations, these questions about random graphs lead to a class of highly nontrivial variational problems, which are successfully analyzed in~\cite{LubZha17,BhaGanLubZha17}. Some related applications which require more than two-point alphabets are described by Yan in~\cite{Yan--nldp}.  Applications with two-point alphabets are treated again by Eldan and Gross in~\cite{Eld--GWcomp,EldGro18,EldGro--exprndmgraphs}, where they obtain refinements of some of the Chatterjee--Dembo estimates using Eldan's new structural results.  Another application discussed in~\cite{ChaDem16} is to large deviations of the number of arithmetic progressions in a random arithmetic set, and the recent paper~\cite{BhaGanShaZha--APs} gives a more complete analysis of this problem using Eldan's approach.

In this paper we do not investigate whether our new results lead to any further improvements in those applications.  In the case of large deviations of subgraph counts in Erd\H{o}s--R\'enyi random graphs, the recent preprint~\cite{DemCoo--subgraphuppertail}
achieves substantial improvements over the other works cited above by carefully finding and exploiting certain convex sets related to that problem.

Another variation on Chatterjee and Dembo's original partition-function estimates appears in Augeri's recent preprint~\cite{Aug--NLDP}.  She considers an arbitrary compactly supported probability measure $\l$ on $\bbR^n$ and a continuously differentiable potential function $f:\bbR^n\to\bbR$.  By an elegant application of elementary convex analysis, she obtains a new upper bound on the partition function $\log \int \rme^f\,\d\l$ in terms of a covering-number estimate on the range of the derivative $Df$.  From this upper bound she then derives improved estimates in several of the applications of nonlinear large deviations, including again large deviations for cycle counts in Erd\H{o}s--R\'enyi random graphs.

\subsection*{Acknowledgements}

I am grateful to Sourav Chatterjee, Amir Dembo and Ofer Zeitouni for some insightful conversations.  Along with Ronen Eldan and an anonymous referee, they also made valuable suggestions about earlier versions of this paper.

\section{The four principles in the proof of Theorem A}\label{sec:princs}

Our proof of Theorem A combines four basic principles from probability and information theory.  This section recalls these in turn, and the next section assembles them into the proof.

\subsection{A modified chain rule for Kullback--Leibler divergence}\label{subs:chain}

We assume familiarity with the basic properties of Shannon entropy and Kullback--Leibler divergence.  A standard reference in the setting of discrete probability distributions is~\cite[Chapter 2]{CovTho06}.  The KL divergence $\rmD(\mu\,\|\,\g)$ can be defined for any pair of probability measures $\mu$, $\g$ on a general measurable space: it is $+\infty$ unless $\mu \ll \g$, and in that case it is defined by
\[\rmD(\mu\,\|\,\g) := \int \log\frac{\d\mu}{\d\g}\,\d\mu\]
(which may still equal $+\infty$). This generalization and its properties can be found in~\cite[Chapters 2 and 3]{Pin64} or~\cite[Appendix D.3]{DemZei--LDPbook}.

For KL divergence, we need the following modification of the usual chain rule.

\begin{lem}[Modified chain rule]\label{lem:mod-chain}
	For any measurable space $K$, finite measurable partition $\P$ of $K$, and probability measures $\mu$ and $\g$ on $K$, we have
	\begin{equation}\label{eq:mod-chain}
		\rmD(\mu\,\|\,\g) = -\rmH_\mu(\P) + \sum_{P \in \P}\mu(P)\rmD(\mu_{|P}\,\|\,\g).
	\end{equation}
\end{lem}

\begin{proof}
	The usual chain rule for KL divergence gives
	\begin{equation}\label{eq:usual-chain}
	\rmD(\mu\,\|\,\g) = \rmD([\mu]_\P\,\|\,[\g]_\P) + \sum_P\mu(P)\rmD(\mu_{|P}\,\|\,\g_{|P}),
	\end{equation}
	where $[\mu]_\P$ denotes the stochastic vector $(\mu(P))_{P\in\P}$. Now observe that
	\begin{multline*}
		\rmD(\mu_{|P}\,\|\,\g_{|P}) = \int_P \log\frac{\d\mu_{|P}}{\d\g_{|P}}\,\d\mu_{|P} \\= \int_P \log \Big(\g(P)\frac{\d\mu_{|P}}{\d\g}\Big)\,\d\mu_{|P} = \rmD(\mu_{|P}\,\|\,\g) + \log \g(P).
	\end{multline*}
	Inserting this into~\eqref{eq:usual-chain}, we are left with the second right-hand term of~\eqref{eq:mod-chain}, together with the quantity
	\begin{align*}
		\rmD([\mu]_\P\,\|\,[\g]_\P) + \sum_P\mu(P) \log\g(P) &= \sum_P\mu(P)\log\frac{\mu(P)}{\g(P)} + \sum_P\mu(P) \log\g(P) \\
		&= - \rmH_\mu(\P).
	\end{align*}
\end{proof}

\subsection{Gibbs' variational principle}\label{subs:Gibbs}

Fix a probability space $(K,\l)$ and let $f:K\to \bbR$ be a bounded measurable function.  As in the Introduction, the \textbf{Gibbs measure} on $K$ associated to $f$ and $\l$ is defined by
\begin{equation}\label{eq:Gibbs}
\mu(\d x) := \frac{\rme^{f(x)}}{\int \rme^f\,\d\l}\l(\d x).
\end{equation}

The importance of Gibbs measures is intimately related to Gibbs' famous variational principle.  Here we use it in the following form.

\begin{prop}\label{prop:Gibbs-with-error}
If $\mu$ is the Gibbs measure associated to $f$, then
\[\rmD(\mu\,\|\,\l) - \int f\,\d\mu = -\log \int \rme^f\,\d\l.\]
	For any other probability measure $\nu$ on $K$, we have
\begin{equation}\label{eq:Gibbs-with-error}
\Big[\rmD(\nu\,\|\,\l) - \int f\,\d\nu\Big] = \Big[\rmD(\mu\,\|\,\l) - \int f\,\d\mu\Big] + \rmD(\nu\,\|\,\mu).
\end{equation}
	\end{prop}

If $K$ is finite and $\l$ is uniform, then~\eqref{eq:Gibbs-with-error} appears within the calculation in~\cite[equations (12.5)--(12.12)]{CovTho06}.  The same calculation gives the general case upon replacing $\rmH(\,\cdot\,)$ with $-\rmD(\,\cdot\,\|\,\l)$ throughout.  In that generality it can be found in the proof of~\cite[Lemma 6.2.13]{DemZei--LDPbook}, and in the application of~\cite[equation~(2.6)]{Csi75} in Section 3, case (A) of that paper.

\subsection{Marton's transportation-entropy inequality}\label{subs:Marton}

A classical inequality of Marton~\cite{Mar86,Mar96} provides the basic link between KL divergence relative to a product measure and the transportation metric~\eqref{eq:trans}.  Most of the proof of Theorem A concerns information theoretic estimates, before Marton's inequality turns these into a transportation-distance bound at the last step.

\begin{prop}[Marton's transportation-entropy inequality]\label{prop:Marton}
	Let $(K_i,d_{K_i})$, $i=1,2,\dots,n$, be complete and separable metric spaces of diameter at most $1$.  If $\mu,\nu \in \Pr(\prod_i K_i)$ and $\nu$ is a product measure then
	\[\ol{d_n}(\mu,\nu) \leq \sqrt{\frac{1}{2n}\rmD(\mu\,\|\,\nu)}.\]
	\qed
\end{prop}

Marton's original presentation is only for measures on finite sets, but the proof works without change for products of general complete separable metric spaces. The only requirement is that they all have diameter at most one: this is why we assume that from the beginning of this paper.  (If the diameters are larger but finite, this simply changes the factor of $1/2$ in the inequality.)  See~\cite{Mar96,Mar98,Dem97,Sam00} for various extensions of Proposition~\ref{prop:Marton}, and~\cite{Tal96b} for a Gaussian analog.

\subsection{DTC and a version of Han's inequality}\label{subs:Han}

Given finite-valued random variables $\xi_1$, \dots, $\xi_n$, their \textbf{dual total correlation} or \textbf{DTC} is
\[\DTC(\xi_1,\dots,\xi_n):= \rmH(\xi_1,\dots,\xi_n) - \sum_{i=1}^n \rmH(\xi_i\,|\,\xi_1,\dots,\xi_{i-1},\xi_{i+1},\dots,\xi_n).\]
For more general random variables valued in measurable spaces $K_1$, \dots, $K_n$, their DTC is
\[\DTC(\xi_1,\dots,\xi_n) := \sup_{\P_1,\dots,\P_n}\DTC([\xi_1]_{\P_1},\dots,[\xi_n]_{\P_n}),\]
where the supremum runs over all tuples of finite measurable partitions $\P_i$ of the spaces $K_i$, and $[\xi_i]_{\P_i}$ denotes the quantization of $\xi_i$ by $\P_i$.

DTC is one of several possible ways to quantify the correlation among $n$ random variables.  A classical family of inequalities due to Han~\cite[Theorem 4.1]{Han78} includes the fact that DTC is always non-negative.  DTC is studied carefully in~\cite{Aus--TC+DTC}, together with its simpler relative TC.  In this paper we use DTC via an alternative formula in the setting of Gibbs measures.

Consider the spaces $K_i$ and reference measures $\l_i \in \Pr(K_i)$ as in the Introduction. The next lemma is~\cite[Proposition 6.1, part (b)]{Aus--TC+DTC}, along with the first remark following its proof in that paper.

\begin{lem}\label{lem:DTC-and-Han}
If $\rmD(\mu\,\|\,\l_1\times \cdots \times \l_n) < \infty$ then
\[\DTC(\mu) = \sum_{i=1}^n \int \rmD(\mu_{i,\bf{z}}\,\|\,\l_i)\,\mu_{[n]\setminus i}(\d \bf{z}) - \rmD(\mu\,\|\,\l_1\times \cdots \times \l_n),\]
where
\begin{itemize}
	\item $\mu_{[n]\setminus i}$ is the projection of $\mu$ to $\prod_{j \in [n]\setminus i}K_j$, and
	\item $(\mu_{i,\bf{z}}:\ \bf{z} \in \prod_{j \in [n]\setminus i}K_j)$ is a conditional distribution under $\mu$ of the $i^{\rm{th}}$ coordinate given the other coordinates. \qed
\end{itemize}
\end{lem}

In the form given by this lemma, the non-negativity of DTC is a sharpening of the usual logarithmic Sobolev inequality for product spaces.  Indeed, this sharpening already appears implicitly inside standard proofs of that logarithmic Sobolev inequality: see the discussion at the end of~\cite[Section 6]{Aus--TC+DTC}.

In case $\mu$ is the Gibbs measure of $f$, Lemma~\ref{lem:DTC-and-Han} turns into a simple expression for $\DTC(\mu)$ in terms of $\nabla f$.  It plays a crucial role later in this paper.

\begin{cor}\label{cor:DTC-of-Gibbs}
The Gibbs measure in~\eqref{eq:Gibbs-main} satisfies
\begin{equation}\label{eq:DTC-of-Gibbs}
\DTC(\mu) = \int \rmD(\xi_{\nabla f(\bf{x},\,\cdot\,)}\,\|\,\l_1\times \cdots \times \l_n)\,\mu(\d \bf{x}) - \rmD(\mu\,\|\,\l_1\times \cdots \times \l_n).
\end{equation}
\end{cor}

\begin{proof}
The additivity of KL divergence for product measures gives
\[\int \rmD(\xi_{\nabla f(\bf{x},\,\cdot\,)}\,\|\,\l_1\times \cdots \times \l_n)\,\mu(\d \bf{x}) = \sum_{i=1}^n \int \rmD(\xi_{\partial_i f(\bf{x},\,\cdot\,)}\,\|\,\l_i)\,\mu(\d \bf{x}).\]
Using this, the right-hand side of~\eqref{eq:DTC-of-Gibbs} matches the formula for $\DTC(\mu)$ in Lemma~\ref{lem:DTC-and-Han}, because $\xi_{\partial_i f(\bf{x},\,\cdot\,)}$ is precisely the conditional distribution of $x_i$ under $\mu$ given the other coordinates $\bf{x}_{[n]\setminus i}$.
	\end{proof}

In the proof of Theorem A, the fact we really need is that the difference on the right-hand side of~\eqref{eq:DTC-of-Gibbs} is non-negative.  However, Theorem A also provides a bound on $\DTC(\mu)$.  This actually offers an alternative route to a decomposition of $\mu$ into near-product measures: the main result of~\cite{Aus--TC+DTC} obtains such a decomposition precisely from a bound on $\DTC$.  We discuss this further at the end of Section~\ref{sec:A-proof}.

\section{Proof of Theorem A}\label{sec:A-proof}

We now return to the setting of Theorem A.

The measure $\mu$ and discrete gradient $\nabla f$ are both constructed from the potential function $f$.  The proof of Theorem A depends on the following link between them.

\begin{lem}\label{lem:integrals-equal}
	We have
	\[\int \nabla f(\bf{x},\bf{x})\,\mu(\d \bf{x}) = \iint \nabla f(\bf{x},\bf{y})\ \xi_{\nabla f(\bf{x},\,\cdot\,)}(\d \bf{y})\,\mu(\d \bf{x}).\]
	\end{lem}

\begin{proof}
	It suffices to prove that
	\begin{equation}\label{eq:integrals-equal-basic}
	\int \partial_if(\bf{x},x_i)\,\mu(\d \bf{x}) = \iint \partial_i f(\bf{x},y)\ \xi_{\partial_if(\bf{x},\,\cdot\,)}(\d y)\,\mu(\d \bf{x})
	\end{equation}
for each $i\in [n]$, for then we can just sum over $i$. The function $\partial_if(\bf{x},\,\cdot\,)$ depends only on $\bf{x}_{[n]\setminus i}$, and $\xi_{\partial_if(\bf{x},\,\cdot\,)}$ is the conditional distribution of $x_i$ under $\mu$ given the other coordinates $\bf{x}_{[n]\setminus i}$. Therefore
	\[\int \partial_i f(\bf{x},y)\ \xi_{\partial_if(\bf{x},\,\cdot\,)}(\d y) = E_\mu\big[\partial_if(\bf{x},x_i)\,\big|\,\bf{x}_{[n]\setminus i}\big],\]
	and~\eqref{eq:integrals-equal-basic} is the tower property of iterated conditional expectations.
\end{proof}


To prove Theorem A, we must bound both $\DTC(\mu)$ and the expression
\begin{equation}\label{eq:ave-trans}
\sum_P\mu(P)\int \rmD(\mu_{|P}\,\|\,\xi_{\nabla f(\bf{y},\,\cdot\,)})\,\mu_{|P}(\d\bf{y}).
\end{equation}
The proof focuses on this expression, but also yields a bound on $\DTC(\mu)$ as a by-product.  This feature of our work bears a curious resemblance to another part of Eldan's paper~\cite{Eld--GWcomp}.  For a Gibbs measure $\mu$ defined relative to a standard Gaussian distribution on $\bbR^n$, rather than on a hypercube,~\cite[Theorem 4]{Eld--GWcomp} gives an upper bound on the deficit in the Gaussian logarithmic Sobolev inequality satisfied by $\mu$ in terms of the gradient complexity (measured using Gaussian width) of its potential function. In view of the relationship between $\DTC$ and deficits in logarithmic Sobolev inequalities, already remarked following Lemma~\ref{lem:DTC-and-Han} above, our bound on $\DTC(\mu)$ (Theorem A part (a)) could be an analog of that deficit bound in~\cite{Eld--GWcomp}.  We do not explore this connection further here.

To lighten notation during the rest of this section, let us abbreviate
\[\rmD\Big(\nu\,\Big\|\,\prod_{i \in S}\l_i\Big) \quad \hbox{to} \quad \rmD(\nu)\]
whenever $S \subseteq [n]$ and $\nu \in \Pr(\prod_{i\in S} K_i)$.

Before the formal proof, let us give an informal sketch highlighting the relevance of Lemma~\ref{lem:integrals-equal}.  In this sketch, we assume that $K_i = \{0,1\}$ for each $i$.  In this case each discrete gradient $\nabla f(\bf{x},\,\cdot\,)$ may be identified with a vector in $\bbR^n$: its $i^{\rm{th}}$ entry is
\begin{equation}\label{eq:loc-grad-pm}
f(x_1,x_2,\dots,x_{i-1},1,x_{i+1},\dots,x_n) - f(x_1,x_2,\dots,x_{i-1},0,x_{i+1},\dots,x_n).
\end{equation}
Conversely, any vector $\bf{u} \in \bbR^n$ defines a linear function on $\{0,1\}^n$ using the Euclidean inner product: $\bf{u}(\bf{y}) := \sum_i u_iy_i$. If $\bf{u}$ is the vector given by~\eqref{eq:loc-grad-pm}, then $\bf{u}(\bf{y})$ agrees with our earlier definition of $\nabla f(\bf{x},\bf{y})$ if we choose the reference points $\ast_i = 0$ for each $i$.

To simplify this sketch further, let us also imagine that the discrete gradient $\nabla f(\bf{x},\,\cdot\,)$ takes only two distinct values in $\bbR^n$, say $\bf{u}$ and $\bf{v}$. (This is really a fantasy: a short exercise shows that, if $f$ is not linear itself, then its discrete gradient function takes at least four distinct values.  But the story is simplest if we pretend there are just two.)

Under these assumptions, we take the partition $\cal{P}$ to consist of
\[P := \{\bf{x}:\ \nabla f(\bf{x},\,\cdot\,) = \bf{u}\} \quad \hbox{and}\quad P^\rm{c} := \{\bf{x}:\ \nabla f(\bf{x},\,\cdot\,) = \bf{v}\}.\]
Assume further that neither of the values $\mu(P)$, $\mu(P^\rm{c})$ is close to zero.  (If one of them is close to zero, then a slightly degenerate version of the ensuing discussion applies.)  So we need to show that $\rmD(\mu_{|P}\,\|\,\xi_{\bf{u}})$ and $\rmD(\mu_{|P^{\rm{c}}}\,\|\,\xi_{\bf{v}})$ are both small relative to $n$.

Here are the steps in the proof:
\begin{itemize}
\item Gibbs' variational principle lets us interpret this requirement another way: $\mu_{|P}$ must come close to saturating the variational inequality that characterizes the product measure $\xi_{\bf{u}}$, meaning that
	\begin{equation}\label{eq:var-gap}
\rmD(\mu_{|P}) - \int \bf{u}\ \d\mu_{|P} \quad \hbox{is not much larger than}\quad \rmD(\xi_{\bf{u}}) - \int \bf{u}\ \d\xi_{\bf{u}}
\end{equation}
relative to $n$, and similarly when comparing $\mu_{|P^{\rm{c}}}$ with $\xi_{\bf{v}}$.
\item For each of $\mu_{|P}$ and $\mu_{|P^\rm{c}}$ separately, we have no clear way to control the gap in~\eqref{eq:var-gap}.  But we can control the average of those two gaps over $P$ and $P^{\rm{c}}$, keeping in mind that $\nabla f(\bf{x},\,\cdot\,)$ equals $\bf{u}$ on $P$ and $\bf{v}$ on $P^{\rm{c}}$ respectively.  By Lemma~\ref{lem:mod-chain}, the average of the KL divergences is
\[\mu(P)\cdot \rmD(\mu_{|P}) + \mu(P^\rm{c})\cdot \rmD(\mu_{|P^\rm{c}}) = \rmD(\mu) + \rmH_\mu(\P),\]
and by Corollary~\ref{cor:DTC-of-Gibbs} this is equal to
\begin{align*}
&\Big[\int\rmD(\xi_{\nabla f(\bf{x},\,\cdot\,)})\,\mu(\d\bf{x}) - \DTC(\mu)\Big] + \rmH_\mu(\P)\\
&= \mu(P)\cdot \rmD(\xi_{\bf{u}}) + \mu(P^{\rm{c}})\cdot \rmD(\xi_{\bf{v}}) - \DTC(\mu) + \rmH_\mu(\P).
\end{align*}
The average of the required integrals is
\begin{align*}
&\mu(P) \int \bf{u}\ \d\mu_{|P} + \mu(P^\rm{c})\int \bf{v}\ \d\mu_{|P^\rm{c}}\\ &=\int_P \bf{u}\ \d\mu + \int_{P^{\rm{c}}} \bf{v}\ \d\mu\\ &= \int \nabla f(\bf{x},\bf{x})\ \mu(\d\bf{x})\\
&= \iint \nabla f(\bf{x},\bf{y})\ \xi_{\nabla f(\bf{x},\,\cdot\,)}(\d\bf{y})\, \mu(\d\bf{x})\\
&= \mu(P) \int \bf{u}(\bf{y})\ \xi_{\bf{u}}(\d\bf{y}) + \mu(P^\rm{c})\int \bf{v}(\bf{y})\ \xi_{\bf{v}}(\d\bf{y}).
\end{align*}
The equality of the third and fourth lines here is the crucial appearance of Lemma~\ref{lem:integrals-equal}.  Combining these calculations of averages, we arrive at
\begin{align}\label{eq:pre-Dem-id}
&\mu(P)\cdot \rmD(\mu_{|P}\,\|\,\xi_{\bf{u}}) + \mu(P^\rm{c})\cdot \rmD(\mu_{|P^\rm{c}}\,\|\,\xi_{\bf{v}}) \nonumber\\
&= \mu(P)\Big[\rmD(\mu_{|P}) - \int \bf{u}\ \d\mu_{|P}\Big] -\mu(P)\Big[\rmD(\xi_{\bf{u}}) - \int \bf{u}\ \d\xi_{\bf{u}}\Big] \nonumber\\
&\qquad + \mu(P^\rm{c})\Big[\rmD(\mu_{|P^\rm{c}}) - \int \bf{v}\ \d\mu_{|P^\rm{c}}\Big] - \mu(P^\rm{c})\Big[\rmD(\xi_{\bf{v}}) - \int \bf{v}\ \d\xi_{\bf{v}}\Big] \nonumber\\
&= -\DTC(\mu) + \rmH_\mu(\P).
\end{align}
\item Since $\DTC$ is non-negative, the last line above is at most $\rmH_\mu(\P) \leq \log|\P| = \log 2$.  This is indeed small compared to $n$, and hence so are both $\rmD(\mu_{|P}\,\|\,\xi_{\bf{u}})$ and $\rmD(\mu_{|P^{\rm{c}}}\,\|\,\xi_{\bf{v}})$.  Moreover, the last line above cannot be negative, since it is a positive combination of differences in the variational principle. So we also find that $\DTC(\mu) \leq \log 2$.
\end{itemize}

Now we give the careful proof in full.

\begin{proof}[Proof of Theorem A]
Start by considering a single $P\in \P$.  For any additively separable function $g$, Proposition~\ref{prop:Gibbs-with-error} gives
\[	\rmD(\mu_{|P}\,\|\,\xi_g) = \Big[\rmD(\mu_{|P}) - \int g\,\d\mu_{|P}\Big] - \Big[\rmD(\xi_g) - \int g\,\d\xi_g\Big].\]
Let us average this identity over $g = \nabla f(\bf{x},\,\cdot\,)$ where $\bf{x}\sim \mu_{|P}$.  The result is
\begin{align*}
	&\int \rmD(\mu_{|P}\,\|\,\xi_{\nabla f(\bf{x},\,\cdot\,)})\,\mu_{|P}(\d \bf{x}) \nonumber \\
	&=\rmD(\mu_{|P}) - \iint \nabla f(\bf{x},\bf{y})\ \mu_{|P}(\d \bf{y})\,\mu_{|P}(\d \bf{x}) \\
	&\qquad -\int\rmD(\xi_{\nabla f(\bf{x},\,\cdot\,)})\,\mu_{|P}(\d \bf{x}) + \iint \nabla f(\bf{x},\bf{y})\ \xi_{\nabla f(\bf{x},\,\cdot\,)}(\d\bf{y})\,\mu_{|P}(\d \bf{x}).
\end{align*}
Now we average this equality over $P \in \cal{P}$ with the weights $\mu(P)$.  The third right-hand term simplifies according to the law of total probability:
\[\sum_P\mu(P)\int\rmD(\xi_{\nabla f(\bf{x},\,\cdot\,)})\,\mu_{|P}(\d \bf{x}) = \int\rmD(\xi_{\nabla f(\bf{x},\,\cdot\,)})\,\mu(\d \bf{x}).\]
The fourth right-hand term simplifies similarly. After these simplifications we are left with
	\begin{align}\label{eq:pre-B}
	&\sum_P\mu(P)\int \rmD(\mu_{|P}\,\|\,\xi_{\nabla f(\bf{y},\,\cdot\,)})\,\mu_{|P}(\d\bf{y}) \nonumber \\
	&= \sum_P\mu(P)\cdot\rmD(\mu_{|P}) - \sum_P\mu(P)\iint \nabla f(\bf{x},\bf{y})\,\mu_{|P}(\d\bf{y})\,\mu_{|P}(\d\bf{x}) \nonumber\\
	&\quad -\int \rmD(\xi_{\nabla f(\bf{x},\,\cdot\,)})\,\mu(\d\bf{x}) + \iint \nabla f(\bf{x},\bf{y})\,\xi_{\nabla f(\bf{x},\,\cdot\,)}(\d \bf{y})\,\mu(\d\bf{x}).
\end{align}

Next we re-write the right-hand side of~\eqref{eq:pre-B} by considering separately (i) the KL-divergence terms and (ii) the double-integral terms:
\begin{itemize}
\item[(i)] By Lemma~\ref{lem:mod-chain} and Corollary~\ref{cor:DTC-of-Gibbs}, we have
\begin{multline*}
\sum_P\mu(P)\cdot\rmD(\mu_{|P}) - \int \rmD(\xi_{\nabla f(\bf{x},\,\cdot\,)})\,\mu(\d\bf{x})\\ = \big[\rmD(\mu) + \rmH_\mu(\P)\big] - \big[\rmD(\mu) + \DTC(\mu)\big] = \rmH_\mu(\P) - \DTC(\mu).
\end{multline*}

\item[(ii)] Using Lemma~\ref{lem:integrals-equal} to substitute for the second double integral in~\eqref{eq:pre-B}, and then using the law of total probability, the difference of those double-integral terms is equal to
\begin{multline*}
\sum_P\mu(P)\Big[-\iint \nabla f(\bf{x},\bf{y})\,\mu_{|P}(\d\bf{y})\,\mu_{|P}(\d\bf{x}) + \int \nabla f(\bf{y},\bf{y})\,\mu_{|P}(\d\bf{y})\Big]\\
= \sum_P\mu(P)\iint \big(\nabla f(\bf{y},\bf{y}) - \nabla f(\bf{x},\bf{y})\big)\,\mu_{|P}(\d\bf{y})\,\mu_{|P}(\d\bf{x}).
\end{multline*}
\end{itemize}

Inserting these calculations into~\eqref{eq:pre-B} and re-arranging slightly, we arrive at the identity
\begin{multline}\label{eq:Dem-id}
\DTC(\mu) + \sum_P \mu(P)\int \rmD(\mu_{|P}\,\|\,\xi_{\nabla f(\bf{y},\,\cdot\,)})\,\mu_{|P}(\d\bf{y}) \\ = \rmH_\mu(\P) + \sum_P \mu(P)\iint \big(\nabla f(\bf{y},\bf{y}) - \nabla f(\bf{x},\bf{y})\big)\,\mu_{|P}(\d\bf{y})\,\mu_{|P}(\d \bf{x}).
\end{multline}
This identity generalizes equation~\eqref{eq:pre-Dem-id} in the proof-sketch above.

By our assumption~\eqref{eq:P-small}, we have
\begin{equation*}
|\nabla f(\bf{y},\bf{y}) - \nabla f(\bf{x},\bf{y})| < \delta n
\end{equation*}
whenever $\bf{x}$ and $\bf{y}$ lie in the same cell of $\P$. This implies that the average of double integrals on the right-hand side of~\eqref{eq:Dem-id} is less than $\delta n$, and so
\begin{equation}\label{eq:pre-id}
	\DTC(\mu) + \sum_P\mu(P)\int \rmD(\mu_{|P}\,\|\,\xi_{\nabla f(\bf{y},\,\cdot\,)})\,\mu_{|P}(\d\bf{y}) < \rmH_\mu(\P) + \delta n.
\end{equation}
	Both terms on the left-hand side of~\eqref{eq:pre-id} are non-negative, so conclusions (a) and (b) of Theorem A follow immediately. Conclusion (b) implies conclusion (c) using Marton's inequality (Proposition~\ref{prop:Marton}) and H\"older's inequality:
\begin{align*}
	&\sum_P\mu(P)\int \ol{d_n}(\mu_{|P},\xi_{\nabla f(\bf{y},\,\cdot\,)})\,\mu_{|P}(\d\bf{y}) \\ &\leq \sum_P\mu(P)\int \sqrt{\frac{1}{2n}\rmD(\mu_{|P}\,\|\,\xi_{\nabla f(\bf{y},\,\cdot\,)})}\,\mu_{|P}(\d\bf{y}) \\
	&\leq \sqrt{\sum_P\mu(P)\int \frac{1}{2n}\rmD(\mu_{|P}\,\|\,\xi_{\nabla f(\bf{y},\,\cdot\,)})\,\mu_{|P}(\d\bf{y})}.
\end{align*}
	\end{proof}

The proof above is quite insensitive to the structure of the spaces $K_i$.  When $K_i = \{-1,1\}$, it yields similar results to those obtained from Eldan's diffusion approach.  But one can imagine `even smaller' marginal spaces:
	
\begin{ques}
	Can the structural results be improved when $K_i = \{-1,1\}$ and the reference measures $\l_i$ are highly biased, for instance when
	\[\l_i\{1\} = p \ll 1 \quad \hbox{for each}\ i?\]
	\end{ques}

Next let us fill in the proof of Corollary A$'$ from Theorem A.

\begin{proof}[Proof of Corollary A$'$ from Theorem A]
	Let $\Q$ be a partition of
\[\Big\{\nabla f(\bf{x},\,\cdot\,):\ \bf{x} \in \prod_i K_i\Big\}\]
into sets of $\|\cdot\|$-diameter less than $\delta n$, and let $\P$ be the pullback of $\Q$ under the map $\bf{x}\mapsto \nabla f(\bf{x},\,\cdot\,)$. Now pick an element $\bf{y}_P$ in each $P\in \P$ which minimizes $\ol{d_n}(\mu_{|P},\xi_{\nabla f(\bf{y}_P,\,\cdot\,)})$, and set $g_P := \nabla f(\bf{y}_P,\,\cdot\,)$.  The first desired inequality follows from part (b) of Theorem A, and then the second follows by the same use of Proposition~\ref{prop:Marton} as above.
\end{proof}

\subsection{Possible variations on Theorem A}

The proof of Theorem A pivots around the identity~\eqref{eq:Dem-id}.  As Amir Dembo has emphasized to me, this identity may have other valuable consequences.  The double integral on the right seems to deserve particular attention.  In the proof above we simply use a uniform bound on the integrand $\nabla f(\bf{y},\bf{y}) - \nabla f(\bf{x},\bf{y})$.  Are there other ways to bound this double integral, perhaps by exploiting further the special structure of the Gibbs measure $\mu$?  The reward could be a result similar to Theorem A but with weaker hypotheses.

In a separate direction, one could try replacing the use of Marton's inequality (Proposition~\ref{prop:Marton}) in Theorem A with a different transportation-entropy inequality.  In~\cite{Dem97}, Dembo proves several analogs of Proposition~\ref{prop:Marton} that replace $\ol{d_n}$ with various alternative transportation-like quantities.  Those analogs then recover some powerful variations on product-space measure concentration that were previously discovered and used by Talagrand~\cite{Tal95}.  It might be worth exploring a version of Theorem A part (c) for one of those alternative transportation-like quantities.

Further possibilities arise if we know more about the coordinate spaces $K_i$.  We have already mentioned the case of two-point spaces in connection with the foundational papers~\cite{ChaDem16} and~\cite{Eld--GWcomp}, but other special examples are also worth considering. For example, if $K_i = [0,1]$ for each $i$ and if $f:[0,1]^n\to\bbR$ is differentiable, then we can imagine using its gradient $Df$ in the usual sense of calculus as a substitute for our discrete gradient $\nabla f$. Regarding $Df$ as a function from $[0,1]^n$ to $\bbR^n$, we can still measure its `complexity' in terms of covering numbers, and seek a description of the corresponding Gibbs measures $\mu$ if this `complexity' is small enough. Some of the steps above should adapt with little change, but a few sticking points stand out.  Perhaps most important is Lemma~\ref{lem:integrals-equal}, which depends on recognizing $\xi_{\partial_i(\bf{x}\,\,\cdot\,)}$ as the conditional distribution of $x_i$ under $\mu$ given the other coordinates $\bf{x}_{[n]\setminus i}$.  This identification does not carry over simply to a more `localized' notion of gradient than $\nabla f$.  As a result, Lemma~\ref{lem:integrals-equal} would have to be replaced with a different calculation, probably involving some extra terms whose control requires new methods or assumptions (perhaps on the second derivatives of $f$, by analogy with some of the estimates in~\cite{ChaDem16}).

\subsection{An alternative approach using the DTC bound}

We have seen that conclusion (a) of Theorem A emerges naturally during the course of proving conclusion (b), and that conclusion (b) implies conclusion (c).  However, according to~\cite[Theorem A]{Aus--TC+DTC}, conclusion (a) by itself implies something like conclusion (c).  To be precise, if $\mu$ is a probability measure on a product space $\prod_i K_i$ for which $\DTC(\mu) \leq r^3 n$, then~\cite[Theorem A]{Aus--TC+DTC} asserts that $\mu$ can be a represented as a mixture
\[\mu = \int_L \mu_y\,\nu(\d y)\]
such that (i) the mutual information in the mixture is at most $\DTC(\mu)$ and (ii) there is a measurable family $(\xi_y:\ y\in L)$ of product measures on $\prod_i K_i$ satisfying
\[\int_L \ol{d_n}(\mu_y,\xi_y)\,\nu(\d y) < 2r.\]

In our setting, if we know that
\[\rmH_\mu(\P) + \delta n \leq r^3 n,\]
then this gives such a mixture representation of $\mu$ with mutual information at most $\rmH_\mu(\P) + \delta n$.  This is a softer route with a similar conclusion to Theorem A part (c).  It gives only a mixture with controlled mutual information, rather than a controlled number of terms, but one could then obtain the latter using~\cite[Theorem 9.5]{Aus--WP}.  However, this approach gives weaker estimates than the proof of Theorem A above, which exploits more directly the special nature of our Gibbs measures.

\subsection{Comparison with Eldan's main structure theorem}\label{subs:Eldan-compar}

In this subsection $C$ stands for several universal constants which are not computed explicitly.  The value of $C$ may change from one appearance to the next.

In~\cite{Eld--GWcomp}, Eldan considers a Gibbs measure $\mu$ as in~\eqref{eq:Gibbs-main} on the space $\{-1,1\}^n$. In this case, each discrete gradient $\nabla f(\bf{x},\,\cdot\,)$ may be identified with a vector in $\bbR^n$: this is similar to~\eqref{eq:loc-grad-pm} except using $1$ and $-1$ instead of $1$ and $0$. (To be precise, this differs from Eldan's convention by a factor of $1/2$, which affects a few of the constants below.)  Eldan measures the `complexity' of $f$ relative to linear functions by the Gaussian width of the set of its discrete gradients:
\[\cal{D}(f) := \rm{GW}\big(\big\{\nabla f(\bf{x},\,\cdot\,):\ \bf{x} \in \{-1,1\}^n\big\} \cup \{\bf{0}\}\big).\]
This quantity is called the `Gaussian-width gradient complexity' of $f$.  Eldan's main structure theorem~\cite[Theorem 3]{Eld--GWcomp} represents a Gibbs measure $\mu$ as a mixture of other measures which (i) are mostly close to product measures in $\ol{d_n}$ and (ii) on average carry almost as much entropy as $\mu$ itself, where the quality of these estimates depends on $\cal{D}(f)$.

Let us write $\g := \cal{D}(f)/n$ for convenience.  Specifically, after fixing $\eps > 0$ and also picking an auxiliary parameter $\a > 1$,~\cite[Theorem 3]{Eld--GWcomp} provides a mixture
\[\mu = \int \mu_\theta\,m(\d\theta)\]
over some space of parameters $\theta$ such that
\begin{itemize}
	\item[i)] $\rmH(\mu) - \int \rmH(\mu_\theta)\,m(\d\theta) < C\eps n$ and
	\item[ii)] there is a subset $\Theta$ of values of $\theta$ such that $m(\Theta) > 1 - \frac{1}{n} - \frac{1}{\a}$ and such that every $\theta \in \Theta$ satisfies
	\[\ol{d_n}(\mu_\theta,\xi_{\bf{v}_\theta}) \leq C\sqrt{\a\g/\eps}\]
	for some product measure $\xi_{\bf{v}_\theta}$.
\end{itemize}

Eldan's theorem gives other conclusions as well, such as bounds on the vectors $\bf{v}_\theta$.  The ingredients in Eldan's mixture are `tilts' of $\mu$, whereas in our work they are the conditioned measures $\mu_{|P}$.  We do not explore these features further here, although the special structure of tilts does seem to give an advantage in some of Eldan's applications of his method (particularly the proof of~\cite[Theorem 1]{Eld--GWcomp}).

By Sudakov's minoration~\cite[Theorem 3.18]{LedTal--book} and the inequality $\|\cdot\|_1 \leq \sqrt{n}\|\cdot\|_2$, the covering numbers of a nonempty subset $K \subseteq \bbR^n$ are related to Gaussian width according to
\[\delta \sqrt{n}\sqrt{\log \cov_{\delta n}(K,\|\cdot\|_1)} \leq C\cdot  \rm{GW}(K) \quad \forall \delta > 0. \]
Also, our identification of each discrete gradient with a vector in $\bbR^n$ satisfies $\|\nabla f(\bf{x},\,\cdot\,)\| \leq \|\nabla f(\bf{x},\,\cdot\,)\|_1$, where $\|\,\cdot\,\|$ denotes the uniform norm for functions on $\{-1,1\}^n$ as previously.  Therefore, with $\g$ as above, we have
\begin{equation}\label{eq:GW-and-logcov}
\log \cov_{\delta n}\big(\{\nabla f(\bf{x},\,\cdot\,):\ \bf{x} \in \{-1,1\}^n\},\|\cdot\|\big) \leq C\frac{\cal{D}(f)^2}{\delta^2n} = C\frac{\g^2}{\delta^2}n
\end{equation}
for all $\delta > 0$. For a given $\delta$, this gives~\eqref{eq:bdd-cplx} with $\eps = \g^2/\delta^2$ up to a multiplicative constant.  Thus, assuming a bound on $\g$ does not force a particular choice of $\eps$ and $\delta$ in~\eqref{eq:bdd-cplx}, but sets up a trade-off between them.

After making a choice of $\delta$ in~\eqref{eq:GW-and-logcov} and letting $\eps := \g^2/\delta^2$, the partition in Corollary A$'$ satisfies
\begin{equation}\label{eq:ent-bound}
\rmH(\mu) - \sum_P\mu(P)\rmH(\mu_{|P}) = \rmH_\mu(\P) \leq C\eps n
\end{equation}
for the approximation of entropies, and the conclusion of Corollary A$'$ gives
\begin{equation}\label{eq:dbar-bound}
\sum_P \mu(P)\cdot \ol{d_n}(\mu_{|P},\xi_{g_P}) \leq C\sqrt{\eps + \frac{\g}{\sqrt{\eps}}} \leq C\max\Big\{\sqrt{\eps},\sqrt{\g/\sqrt{\eps}}\Big\}.
\end{equation}
If $\eps \leq \g^{2/3}$ then~\eqref{eq:ent-bound} matches Eldan's conclusion (i) above up to a constant, and the right-hand side of~\eqref{eq:dbar-bound} is $C\sqrt{\g/\sqrt{\eps}}$, which improves on Eldan's conclusion (ii) when $\eps$ is very small.

However, if $\eps > \g^{2/3}$ (equivalently, $\delta < \g^{2/3}$), then the right-hand side of~\eqref{eq:dbar-bound} is $C\sqrt{\eps}$.  In this case we may reduce $\eps$ (equivalently, increase $\delta$) and improve both of the bounds~\eqref{eq:ent-bound} and~\eqref{eq:dbar-bound}. So there is no value in using our estimates with $\eps > \g^{2/3}$.  By contrast, Eldan's bounds are still potentially useful in this range, which enables them to perform better than ours in some applications. We discuss an example at the end of Section~\ref{sec:part-fn-approx}.

This comparison brings up an interesting puzzle:

\begin{ques}
Suppose we begin with a bound on $\g$ and then use~\eqref{eq:GW-and-logcov} to establish the trade-off $\eps = \g^2/\delta^2$.  Can we vary the proof of Theorem A so that there is some benefit to choosing $\delta < \g^{2/3}$?  This would mean that at least one of the left-hand sides of~\eqref{eq:ent-bound} and~\eqref{eq:dbar-bound} is bounded more effectively when $\delta < \g^{2/3}$ than when $\delta \geq \g^{2/3}$.  (I would guess that any improvement must be to~\eqref{eq:dbar-bound}.)
\end{ques}

\section{A characterization of the terms in the mixture}\label{sec:mix-terms}

In this section, we suppose that the discrete gradient function
\[\prod_i K_i \to C_\rm{b}\Big(\prod_i K_i\Big):\bf{x} \mapsto \nabla f(\bf{x},\,\cdot\,)\]
is Lipschitz from $d_n$ to the uniform norm $\|\cdot\|$.  Let
\[L := \sup\Big\{\frac{\|\nabla f(\bf{x},\,\cdot\,) - \nabla f(\bf{y},\,\cdot\,)\|}{d_n(\bf{x},\bf{y})}:\ \bf{x},\bf{y} \in \prod_i K_i\ \rm{distinct}\Big\}\]
be its Lipschitz constant (beware: $L$ is not the Lipschitz constant of $f$ itself).

Let $\cal{P}$ be a partition of $\prod_i K_i$ as in Corollary A$'$, and let $\bf{y}_P \in P$ for $P \in \cal{P}$ be a selection of points that minimize the distances $\ol{d_n}(\mu_{|P},\xi_{\nabla f(\bf{y}_P,\,\cdot\,)})$.  Recall that the functions $g_P$ of Corollary A$'$ are then given by $g_P := \nabla f(\bf{y}_P,\,\cdot\,)$.

\begin{prop}\label{prop:approx-fix-pt}
	These data satisfy
	\[\sum_{P}\mu(P)\Big\|g_P - \int \nabla f(\bf{x},\,\cdot\,)\ \xi_{g_P}(\d\bf{x})\Big\| \leq \delta n + L\sqrt{\frac{1}{2}\Big(\frac{\rmH_\mu(\P)}{n} + \delta \Big)}.\]
	\end{prop}

Assuming that the right-hand side here is small, this means that most of the probability vector $(\mu(P):\ P\in\P)$, when transferred to the list of functions $(g_P:\ P\in \P)$, is on terms that satisfy the `approximate fixed point' equation
\begin{equation}\label{eq:approx-fix}
g_P \approx \int \nabla f(\bf{x},\,\cdot\,)\ \xi_{g_P}(\d\bf{x}).
\end{equation}

\begin{proof}
Since $\bf{y}_P \in P$, our assumption~\eqref{eq:P-small} about $\P$ gives
\[\Big\|g_P - \int \nabla f(\bf{x},\,\cdot\,)\ \mu_{|P}(\d\bf{x})\Big\| < \delta n \quad \forall P \in \cal{P}.\]
On the other hand, the definition of $L$ gives
\[	\Big\|\int \nabla f(\bf{x},\,\cdot\,)\ \mu_{|P}(\d\bf{x}) - \int \nabla f(\bf{x},\,\cdot\,)\ \xi_{g_P}(\d\bf{x})\Big\|
	\leq L\cdot \ol{d_n}(\mu_{|P},\xi_{g_P}) \quad \forall P \in \cal{P}.\]
	Averaging over $\P$ with weights $\mu(P)$, the result follows by conclusion (c) of Theorem A and the triangle inequality.
	\end{proof}

In the special setting of $\{-1,1\}^n$, the approximate equation~\eqref{eq:approx-fix} admits an alternative form which resembles its counterpart in~\cite{EldGro18} more closely.  In that setting, we can identify each discrete gradient $\nabla f(\bf{x},\,\cdot\,)$ with a vector in $\bbR^n$, as previously.  Also, any product measure on $\{-1,1\}^n$ is uniquely specified by its barycentre in $[-1,1]^n$.  Let $\bf{m}_P$ be the barycentre of $\xi_{g_P}$ for each $P$. After these identifications, equation~\eqref{eq:approx-fix} is equivalent to
\begin{equation}\label{eq:approx-fix-EG}
\bf{m}_P \approx \tanh\Big(\int \nabla f(\bf{x},\,\cdot\,)\ \xi_{g_P}(\d \bf{x})\Big) = \tanh\big(D\widetilde{f}(\bf{m}_P)\big),
\end{equation}
where $\widetilde{f}$ is the harmonic extension of $f$ to $[-1,1]^n$ (see~\cite[Subsubsection 3.1.1]{Eld--GWcomp}), $D\widetilde{f}$ is its derivative in the usual sense of calculus, and $\tanh$ is applied coordinate-wise.  Equation~\eqref{eq:approx-fix-EG} is essentially the same as~\cite[equation (8)]{EldGro18}.

Let us return to the formulation in terms of $g_P$ rather than $\bf{m}_P$, but derive from Proposition~\ref{prop:approx-fix-pt} a bound in terms of the Gaussian-width gradient complexity $\cal{D}(f)$ studied by Eldan.  Let $\g:= \cal{D}(f)/n$, as previously; consider an auxiliary parameter $\delta > 0$; and choose a partition $\P$ as promised by inequality~\eqref{eq:GW-and-logcov} for this $\delta$.  Then~\eqref{eq:ent-bound} lets us turn Proposition~\ref{prop:approx-fix-pt} into
	\[\sum_{P}\mu(P)\Big\|g_P - \int \nabla f(\bf{x},\,\cdot\,)\ \xi_{g_P}(\d\bf{x})\Big\| \leq \delta n + CL\sqrt{\frac{\g^2}{\delta^2} + \delta} \quad \forall \delta > 0\]
(where once again $C$ is a universal constant that we do not estimate explicitly, and may change value when it appears again below).

In the regime $L\geq Cn$, which seems to be more relevant to applications, the second right-hand term above is the more significant.  To minimize it, we make the choice $\delta := \g^{2/3}$.  Assuming also that $\g$ is small, we are left with
	\[\sum_{P}\mu(P)\Big\|g_P - \int \nabla f(\bf{x},\,\cdot\,)\ \xi_{g_P}(\d\bf{x})\Big\| \leq \g^{2/3}n + CL\g^{1/3}  \leq CL\g^{1/3}.\]
A result of this flavour is obtained by Eldan and Gross in~\cite[Theorem 9]{EldGro18}.  But their estimates have quite a different shape from ours (for instance, they also involve the Lipschitz constant of $f$ itself, not just that of $\nabla f$), and a direct comparison seems difficult.

\section{Approximation of partition functions}\label{sec:part-fn-approx}

One of the main potential applications of Theorem A is to the estimation of the partition function $\int \rme^f\,\d\l$.  By Proposition~\ref{prop:Gibbs-with-error}, it satisfies
\begin{equation}\label{eq:Gibbs-again}
\log \int \rme^f\,\d\l \geq \int f\,\d\nu - \rmD(\nu\,\|\,\l)
\end{equation}
for any probability measure $\nu$ on $\prod_i K_i$, with equality if and only if $\nu = \mu$, the Gibbs measure associated to $f$.  A structural result like Theorem A allows one to achieve approximate equality using a product measure $\nu$ on the right-hand side.  This can help us estimate the expression on the left.

\begin{prop}\label{prop:part-fn-approx}
	If $f$ is $L$-Lipschitz for the metric $d_n$ and satisfies~\eqref{eq:bdd-cplx}, then
	\begin{equation}\label{eq:part-fn-approx}
	\log \int \rme^f\,\d\l \leq \sup_\xi\Big[\int f\,\d\xi - \rmD(\xi\,\|\,\l)\Big] + (\eps + \delta)n + \sqrt{\frac{\eps + \delta}{2}}L,
	\end{equation}
where the supremum runs over product measures on $\prod_i K_i$.
	\end{prop}

\begin{rmk}
	In case each $K_i$ is finite and $\l_i$ is uniform, standard manipulations turn~\eqref{eq:part-fn-approx} into
\[\log \sum_{\bf{x}}\rme^{f(\bf{x})} \leq \sup_\xi\Big[\rmH(\xi) + \int f\,\d\xi\Big] + (\eps + \delta)n + \sqrt{\frac{\eps + \delta}{2}}L.\]
This is the relevant form for many applications of nonlinear large deviations.
\end{rmk}

\begin{proof}
	Let $\cal{P}$ be the partition implied by~\eqref{eq:bdd-cplx}.  Gibbs' identity and Corollary~\ref{cor:DTC-of-Gibbs} give
	\begin{align}\label{eq:Gibbs-id}
\log \int \rme^f\,\d\l &= \int f\,\d\mu - \rmD(\mu\,\|\,\l) \nonumber \\
&= \int f\,\d\mu + \DTC(\mu) - \int \rmD(\xi_{\nabla f(\bf{y},\,\cdot\,)}\,\|\,\l)\,\mu(\d\bf{y}).
\end{align}
	Using part (c) of Theorem A, we have
	\begin{align}\label{eq:ints-diff}
	\int f\,\d\mu &= \sum_P\mu(P)\int f\,\d\mu_{|P} \nonumber\\
	&\leq \iint f\,\d\xi_{\nabla f(\bf{y},\,\cdot\,)}\,\mu(\d\bf{y}) + L\sum_P\mu(P) \int \ol{d_n}(\mu_{|P},\xi_{\nabla f(\bf{y},\,\cdot\,)})\,\mu_{|P}(\d\bf{y}) \nonumber\\
	&\leq \iint f\,\d\xi_{\nabla f(\bf{y},\,\cdot\,)}\,\mu(\d\bf{y}) + \sqrt{\frac{\eps + \delta}{2}}L.
	\end{align}
Inserting this bound and also part (a) of Theorem A into~\eqref{eq:Gibbs-id}, we obtain
\[\log\int \rme^f\,\d\l \leq \int\Big[\int f\,\d\xi_{\nabla f(\bf{y},\,\cdot\,)} - \rmD(\xi_{\nabla f(\bf{y},\,\cdot\,)}\,\|\,\l)\Big]\,\mu(\d\bf{y}) + (\eps + \delta)n + \sqrt{\frac{\eps + \delta}{2}}L.\]
Now we bound the integral over $\bf{y}$ by the supremum over all product measures $\xi$, as in~\eqref{eq:part-fn-approx}.
	\end{proof}

Partition-function estimation was the heart of the original paper~\cite{ChaDem16} which instigated research into nonlinear large deviations.  A general discussion of the usefulness of such estimates is given in the Introduction to that paper.  In~\cite{Yan--nldp}, Yan extends Chatterjee and Dembo's estimates from the cube $\{0,1\}^n$ to a more general class of products of subsets of Banach spaces.  In the setting of $\{-1,1\}^n$, Eldan shows how partition-function estimates can be derived from his main structure theorem in~\cite[Corollary 2]{Eld--GWcomp}.  One can compare Proposition~\ref{prop:part-fn-approx} with Eldan's results using the ideas from Subsection~\ref{subs:Eldan-compar}, but I have not been able to recover the full strength of Eldan's estimate as a consequence of Theorem A.  This seems to be related to the discussion at the end of Subsection~\ref{subs:Eldan-compar} above.  In the notation of that discussion, the appropriate choice of $\eps$ for the proof of~\cite[Corollary 2]{Eld--GWcomp} is
\[C\Big(\frac{L}{n}\Big)^{2/3}\g^{1/3},\]
which is generally much larger than the threshold $\eps = \g^{2/3}$ that marks the end of the usefulness of our estimates, as discussed in Subsection~\ref{subs:Eldan-compar}.

\bibliographystyle{alpha}
\bibliography{bibfile}






\end{document}